\documentclass[reqno]{amsart}

\usepackage{amsmath,amssymb,amsthm,amsfonts,mathrsfs,textcomp,mathtools,accents}
\usepackage{graphicx}
\usepackage{enumerate}
\usepackage[shortlabels]{enumitem}
\setenumerate{label={\upshape(\roman*)},wide=0pt,topsep=0.2cm,labelsep=0.1em,itemsep=0.15cm,leftmargin=*}
\setlist[1]{topsep=0.2cm,itemsep=0.2cm}

\usepackage[headings]{fullpage}
\usepackage{hyperref}
\hypersetup{
  colorlinks=false
}
\usepackage{bm}
\usepackage{natbib}

\DeclareFontFamily{U}{mathx}{\hyphenchar\font45}
\DeclareFontShape{U}{mathx}{m}{n}{<-> mathx10}{}
\DeclareSymbolFont{mathx}{U}{mathx}{m}{n}
\DeclareMathAccent{\widebar}{0}{mathx}{"73}

\usepackage[figuresright]{rotating}
\usepackage{tikz}
\usetikzlibrary{arrows,calc,positioning}

\newtheorem{theorem}{Theorem}[section]
\newtheorem{corollary}[theorem]{Corollary}
\newtheorem{lemma}[theorem]{Lemma}

\theoremstyle{definition}
\newtheorem{definition}[theorem]{Definition}

\theoremstyle{remark}
\newtheorem{remark}[theorem]{Remark}

\numberwithin{equation}{section}

\newcommand{\dd}{\,{\rm d}}

\newcommand{\R}{\mathbb R}
\DeclareMathOperator{\Ker}{Ker}
\DeclareMathOperator{\Imm}{Imm}
\DeclareMathOperator{\rk}{rk}
\newcommand{\de}{\partial}
\newcommand{\Mat}{{\bf M}}

\begin{document}

\title[The Manifold Of Variations: hazard assessment of short-term
impactors]{The Manifold Of Variations:\\hazard assessment of
  short-term impactors}%

\author{Alessio Del Vigna}%
\address{Space Dynamics Services s.r.l., Via Mario Giuntini,
  Navacchio di Cascina, Pisa, Italy}\email{delvigna@spacedys.com}
\address{Dipartimento di Matematica, Universit\`a di Pisa, Largo Bruno
  Pontecorvo 5, 56127 Pisa, Italy}

\begin{abstract}
    When an asteroid has a few observations over a short time span the
    information contained in the observational arc could be so little
    that a full orbit determination may be not possible. One of the
    methods developed in recent years to overcome this problem is
    based on the systematic ranging and combined with the Admissible
    Region theory to constrain the poorly-determined topocentric range
    and range-rate. The result is a set of orbits compatible with the
    observations, the Manifold Of Variations, a 2-dimensional
    compact manifold parametrized over the Admissible Region. Such a
    set of orbits represents the asteroid confidence region and is
    used for short-term hazard predictions. In this paper we review
    the Manifold Of Variations method and make a detailed analysis of
    the related probabilistic formalism.
\end{abstract}

\maketitle

\section{Introduction}
\label{sec:intro}

An asteroid just discovered has a strongly undetermined orbit, being
it weakly constrained by the few available astrometric
observations. This is called a Very Short Arc (VSA,
\cite{milani2004AR}) since the observations cover a short time
span. As it is well-known, the few observations constrain the position
and velocity of the object in the sky-plane, but leave almost unknown
the distance from the observer and the radial velocity, respectively
the topocentric range and range-rate. This means that the asteroid
confidence region is wide in at least two directions instead of being
elongated and thin. Thus one-dimensional sampling methods, such as the
use of the Line Of Variations (LOV, \cite{milani:multsol}), are not a
reliable representation of the confidence region. Indeed, VSAs are
often too short to allow a full orbit determination, which would make
it impossible even to define the LOV. To overcome the problem three
systems have been developed in recent years: Scout at the NASA JPL
\citep{farnocchia2015}, NEOScan at SpaceDyS and at the University of
Pisa \citep{spoto:immimp,delvigna:phd}, and {\sc neoranger} at the
University of Helsinki \citep{neoranger}.

Scout\footnote{\url{https://cneos.jpl.nasa.gov/scout/}} and
NEOScan\footnote{\url{https://newton.spacedys.com/neodys2/NEOScan/}}
are based on the systematic ranging, an orbit determination method
which explores a suitable subset of the topocentric range and
range-rate space \citep{chesley2005}. They constantly scan the Minor
Planet Center NEO Confirmation Page (NEOCP), with the goal of
identifying asteroids such as NEAs, MBAs or distant objects to
confirm/remove from the NEOCP and to give early warning of imminent
impactors. NEOScan uses a method based on the systematic ranging and
on the Admissible Region (AR) theory \citep{milani2004AR} to include
in the orbit determination process also the information contained in
the short arc, although too little to constrain a full orbit. Starting
from a sampling of the AR, a set of orbits compatible with the
observations is then computed by a doubly constrained differential
correction technique, which in turn ends with a sampling of the
Manifold Of Variations (MOV, \cite{tommei:phd}), a 2-dimensional
manifold of orbits parametrized over the AR and representing a
two-dimensional analogue of the LOV. This combination of techniques
provided a robust short-term orbit determination method and also a
two-dimensional sampling of the confidence region to use for the
subsequent hazard assessment and for the planning of the follow-up
activity of interesting objects.

In this paper we review the algorithm at the basis of NEOScan and
examine in depth the underlying mathematical formalism, focusing on
the short-term impact monitoring and the impact probability
computation. In Section~\ref{sec:armov} we summarise the AR theory and
the main definitions concerning the MOV. Section~\ref{sec:pdf-der}
contains the results for the derivation of the probability density
function on an appropriate integration space, by assuming that the
residuals are normally distributed and by a suitable linearisation. In
Section~\ref{sec:prob_full_nonlin} we show that without the
linearisation assumed in the previous section we find a distribution
known to be inappropriate for the problem of computing the impact
probability of short-term impactors \citep{farnocchia2015}. In
Section~\ref{sec:prob-int} we give a probabilistic interpretation to
the optimisation method used to define the MOV, namely the doubly
constrained differential corrections. Lastly, Section~\ref{sec:conc}
contains our conclusions and a possible application of the method
presented in this paper to study in future research.

\section{The Admissible Region and the Manifold Of Variations}
\label{sec:armov}

As anticipated in the introduction, the systematic ranging method on
which NEOScan is based makes a deep use of the Admissible Region. In
this section we briefly recap its main properties and we refer the
reader to relevant further references.

When in presence of a VSA, even when we are not able to fit a full
least squares orbit, we are anyway able to compute the right ascension
$\alpha$, the declination $\delta$, and their time derivatives
$\dot{\alpha}$ and $\dot{\delta}$, to form the so-called
\emph{attributable} \citep{milani2004AR}. For the notation, we
indicate the attributable as
\[
    \mathcal{A} \coloneqq (\alpha, \delta, \dot{\alpha}, \dot{\delta})
    \in \mathbb{S}^1 \times \left(-\tfrac{\pi}2, \tfrac{\pi}2\right)
    \times \R^2,
\]
where all the quantities are referred to the mean of the observational
times. The AR has been introduced to constrain the possible values of
the range $\rho$ and the range-rate $\dot\rho$ that the attributable
leaves completely undetermined. Among other things, we require that
the observed object is a Solar System body, discarding the orbits
corresponding to too small objects to be source of meteorites (the
so-called shooting star limit).

The AR turns out to be a compact subset of $\R_{\geq 0}\times \R$ and
to have at most two connected components. Commonly the AR has one
component and the case with two components indicates the possibility
for the asteroid to be distant. Since the AR is compact we can sample
it with a finite number of points, and we use two different sampling
techniques depending on the geometric properties of the AR and on the
existence of a reliable least squares orbit. A nominal solution is an
orbit obtained by unconstrained (\emph{i.e.}  full) differential
corrections, starting from a preliminary orbit as first guess.  Indeed
it does happen that the available observations are enough to constrain
a full orbit: it essentially depends on the quality of the observed
arc, which can be measured by the arc type and the signal-to-noise
ratio of the geodesic curvature of the arc on the celestial sphere
\citep{milani2007}. We say that the nominal orbit is \emph{reliable}
if the latter is $>3$. If this is the case we anyway compute a
sampling of orbits compatible with the observations to use to make
predictions, but instead of considering the entire AR we exploit the
knowledge of the least squares orbit, and in particular of its
covariance.\\[0.15cm]
{\em Grid sampling}. When a reliable nominal solution does not exist
or it is not reliable, the systematic ranging is performed by a
two-step procedure and in both steps it uses rectangular grids over
the AR. In the first step a grid covers the entire AR, with the sample
points for $\rho$ spaced either uniformly or with a logarithmic scale
depending on the geometry of the AR (see Figure~\ref{fig:grid},
left). The corresponding sampling of the MOV is then computed, with
the doubly constrained differential corrections procedure described
below. Then we are able to identify the orbits which are more
compatible with the observations by computing their $\chi^2$ value
(see Equation~\eqref{chimov}), and we can restrict the grid to the
region occupied by such orbits. Thus the grid used in the second step
typically covers a smaller region and thus it has a higher resolution
than the first one (see Figure~\ref{fig:grid}, right).
\\[0.5cm]
\noindent{\em Cobweb sampling}. If a reliable nominal orbit exists,
instead of using a grid we compute a spider web sampling in a suitable
neighbourhood of the nominal solution. This is obtained by following
the level curves of the quadratic approximation of the target function
used to minimise the RMS of the observational residuals (see
Figure~\ref{fig:cobweb}).
\\[0.15cm]
\noindent For the formal definition of the AR and the proof of its
properties the reader can refer to \cite{milani2004AR}. The operative
definition used in NEOScan and a detailed explanation of the sampling
techniques can be found in \cite{spoto:immimp} and
\cite{delvigna:phd}.
\begin{figure}[h]
    \includegraphics[scale=0.375]{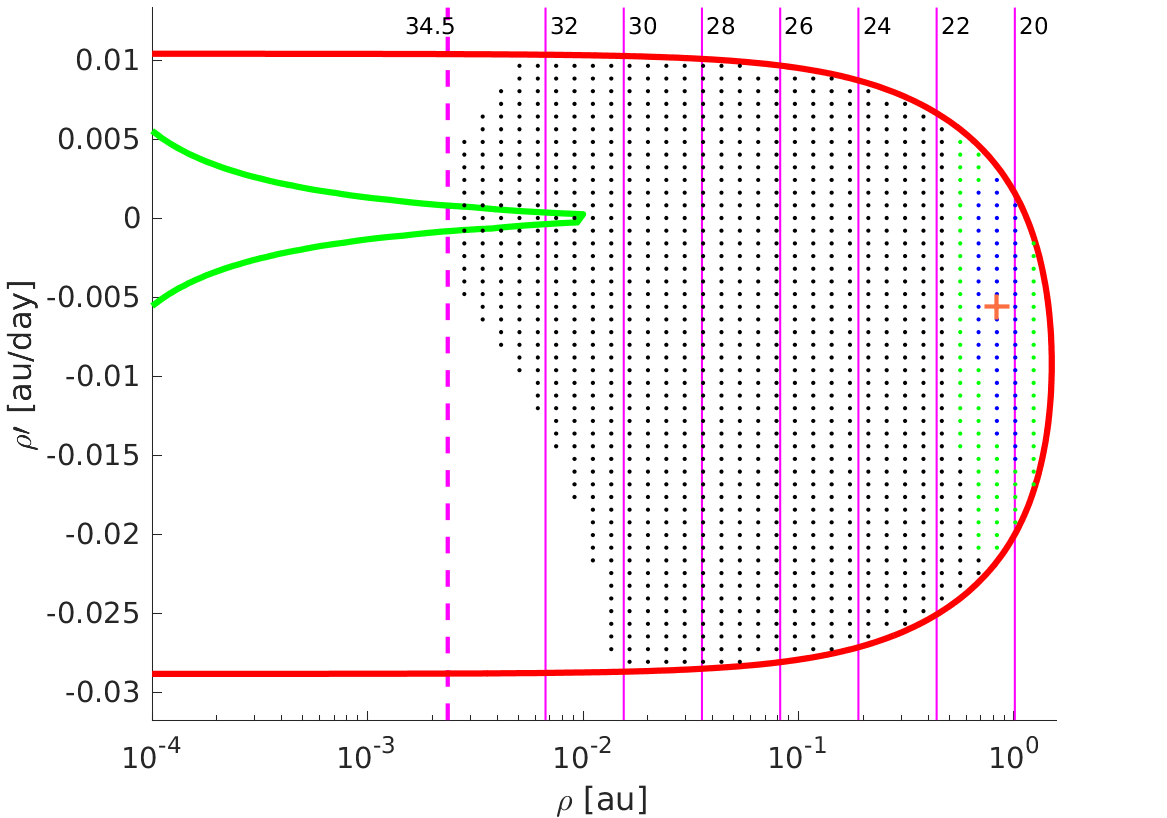}
    \quad
    \includegraphics[scale=0.375]{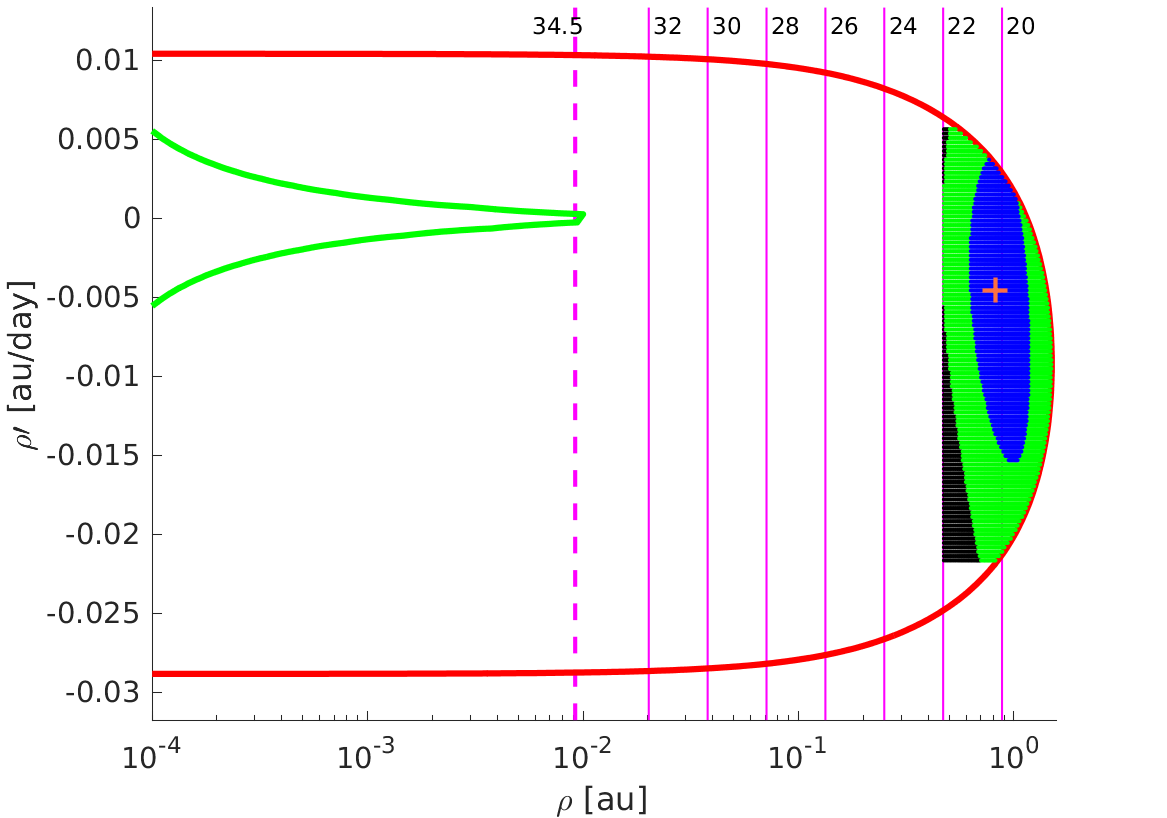}
    \caption{Admissible Region sampling with the rectangular
      grids. \emph{Left}. AR sampling performed with the first grid,
      covering the whole AR. \emph{Right}. Sampling of the region
      containing the orbits with $\chi<5$ resulting from the first
      grid. In both plots the sample points are marked in blue when
      $\chi\leq 2$ and in green when $2<\chi<5$.}
    \label{fig:grid}
\end{figure}
\vspace{-0.85cm}
\begin{figure}[h]
    \includegraphics[scale=0.375]{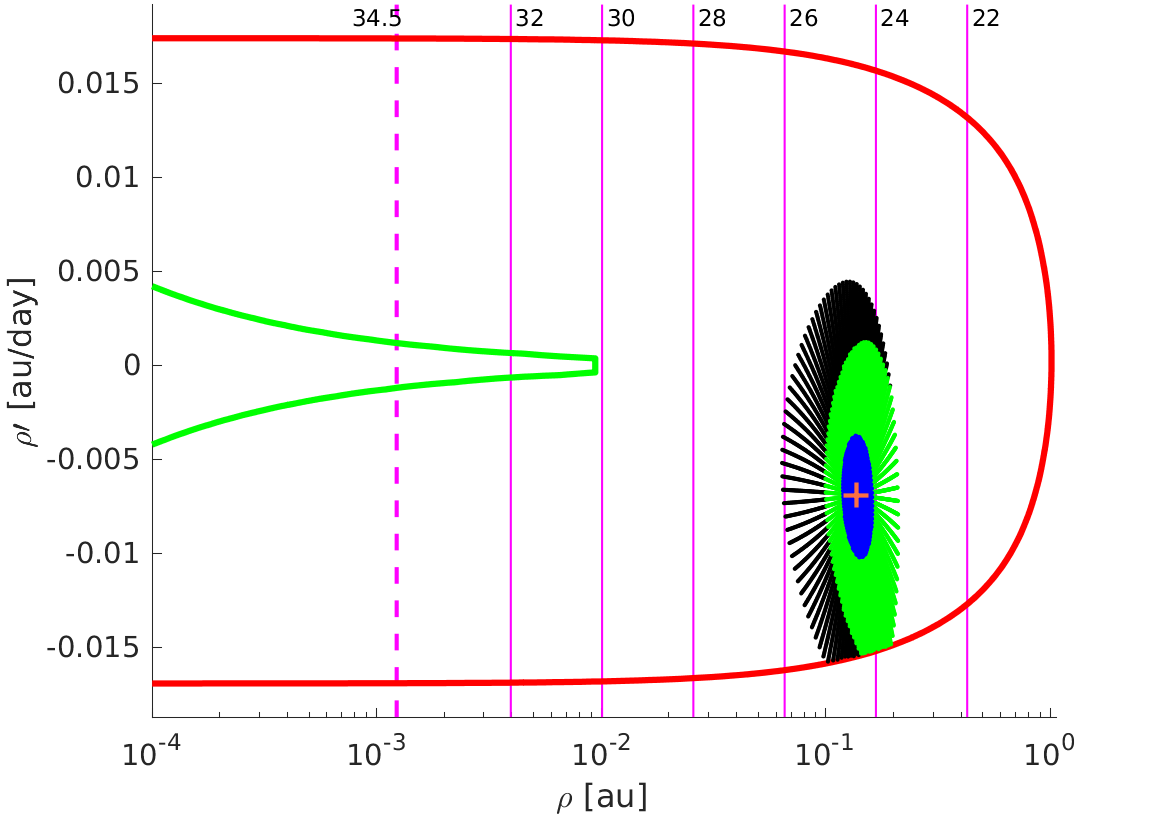}
    \caption{Admissible Region sampling with the cobweb technique. The
      sample points are marked in blue when $\chi\leq 2$ and in green
      when $2<\chi<5$.}
    \label{fig:cobweb}
\end{figure}

We now describe the Manifold Of Variations, a sample of orbits
compatible with the observational data set. In general, to obtain
orbits from observations we use the least squares method. The target
function is
\[
    Q(\mathbf{x}) \coloneqq \frac{1}{m} \bm{\xi}(\mathbf x)^\top W
    \bm{\xi}(\mathbf x),
\]
where $\mathbf{x}$ are the orbital elements to fit, $m$ is the number
of observations used, $\bm{\xi}$ is the vector of the
observed-computed debiased astrometric residuals, and $W$ is the
weight matrix. As explained in the introduction, in general a full
orbit determination is not possible for short arcs. Nevertheless, from
the observational data set it is possible to compute an attributable
$\mathcal{A}_0$. The AR theory has been developed to obtain
constraints on the values of $(\rho,\dot\rho)$, so that we can merge
the information contained in the attributable with the knowledge of an
AR sampling. The basic idea of this method is to fix $\rho$ and
$\dot\rho$ at some specific values $\bm\rho_0=(\rho_0,\dot\rho_0)$
obtained from the AR sampling, compose the full orbit
$(\mathcal{A}_0,\bm\rho_0)$, and fit only the attributable part to the
observations with a suitable differential corrections procedure.

\begin{definition}
    Given a subset $K$ of the AR, we define the \emph{Manifold Of
      Variations} to be the set of points
    $(\mathcal A^*(\bm\rho_0),\bm\rho_0)$ such that $\bm\rho_0\in K$
    and $\mathcal A^*(\bm\rho_0)$ is the local minimum of the function
    $\left.Q(\mathcal{A},\bm\rho)\right|_{\bm\rho=\bm\rho_0}$, when it
    exists. We denote the Manifold Of Variations with $\mathcal M$.
\end{definition}

\begin{remark}
    In general $\mathcal{M}$ is a 2-dimensional manifold, since the
    differential of the map from the $(\rho,\dot\rho)$ space to
    $\mathcal M$ has rank $2$ (see Lemma~\ref{prop:MOVmanifold}).
\end{remark}

\noindent The set $K$ is the intersection of a rectangle with the AR
in the case of the grid sampling, whereas $K$ is an ellipse around
$\bm\rho^*$ in the cobweb case, where
$\mathbf{x}^* = (\mathcal{A}^*,\bm\rho^*)$ is the reliable nominal
orbit. To fit the attributable part we use the \emph{doubly
  constrained differential corrections}, which are usual differential
corrections but performed on a 4-dimensional space rather than on a
6-dimensional one. The normal equation is
\[
    C_{\mathcal A} \Delta \mathcal A = D_{\mathcal A},
\]
where
\begin{equation}\label{C_A}
    C_{\mathcal A} \coloneqq B_{\mathcal A}^\top  W B_{\mathcal A}\,,\quad
    D_{\mathcal A} \coloneqq -B_{\mathcal A}^\top  W \bm{\xi}\,, \quad
    B_{\mathcal A} \coloneqq \frac{\partial{\bm{\xi}}}{\partial{\mathcal A}}.
\end{equation}
We call $K'$ the subset of $K$ on which the doubly constrained
differential corrections converge, giving a point on $\mathcal{M}$.

\begin{definition}
    For each orbit $\mathbf{x}\in \mathcal{M}$ we define the
    \emph{$\chi$-value} to be
    \begin{equation}{\label{chimov}}
        \chi(\mathbf x) \coloneqq \sqrt{m(Q(\mathbf x)- Q^*)},
    \end{equation}
    where $Q^*$ is the minimum value of the target function over $K'$,
    that is
    $Q^* \coloneqq \min_{\bm\rho\in K'}
    Q(\mathcal{A}^*(\bm\rho),\bm\rho)$. The orbit for which this
    minimum is attained is denoted with $\mathbf{x}^*$ and referred to
    as the \emph{best-fitting orbit}.
\end{definition}

The standard differential corrections are a variant of the Newton
method to find the minimum of a multivariate function, the target
function $Q$ \citep{milani:orbdet}. The uncertainty of the result can
be described in terms of confidence ellipsoids (optimisation
interpretation) as well as in the language of probability
(probabilistic interpretation). In Section~\ref{sec:prob-int} we give
a probabilistic interpretation to the doubly constrained differential
corrections procedure.


\section{Derivation of the probability density function}
\label{sec:pdf-der}

Starting from \cite{spoto:immimp} we now give the proper mathematical
formalism to derive the probability density function on a suitable
space $S$ which we define below. Upon integration this will lead to
accurate probability estimates, the most important and urgent of which
is the impact probability computation for a possible imminent
impactor.

Our starting point is a probability density function on the residuals
space, to propagate back to $S$. In particular, we assume the
residuals to be distributed according to a Gaussian random variable
$\bm\Xi$, with zero mean and covariance $\Gamma_{\bm\xi}=W^{-1}$, that
is
\[
    p_{\bm\Xi}(\bm \xi) = N(\bm 0,\Gamma_{\bm \xi})(\bm \xi)=
    \frac{\sqrt{\det W}}{(2\pi)^{m/2}}
    \exp\left(-\frac{mQ(\bm\xi)}{2}\right)=\frac{\sqrt{\det
        W}}{(2\pi)^{m/2}} \exp\left(-\frac 12 \bm\xi^\top
      W\bm\xi\right).
\]
Without loss of generality, we can assume that
\begin{equation}\label{eq:res_norm_gauss}
    p_{\bm\Xi}(\bm \xi) = N(\bm 0,I_m)(\bm \xi)= \frac
    1{(2\pi)^{m/2}}\exp\left(-\frac 12 \bm\xi^\top \bm\xi\right),
\end{equation}
where $I_m$ is the $m\times m$ identity matrix. This is obtained by
using the normalised residuals in place of the true residuals (see for
instance \cite{milani:orbdet}). For the notation, from now on we will
use $\bm\xi$ to indicate the normalised residuals. Note that with the
use of normalised residuals the target function becomes
$Q(\mathbf{x}) = \frac 1m \bm\xi(\mathbf{x})^\top \bm\xi(\mathbf{x})$.

\subsection{Spaces and maps}

Let us introduce the following spaces.
\begin{enumerate}
  \item $S$ is the sampling space, which is $\R_{\geq 0} \times \R$ if
    the sampling is uniform in $\rho$, $\R^2$ if the sampling is
    uniform in $\log_{10}\rho$, and $\R_{\geq 0} \times \mathbb{S}^1$
    in the cobweb case.
  \item $K'\subseteq \R_{\geq 0}\times \R$ has already been introduced
    in Section~\ref{sec:armov}, and it is the subset of the points of
    the AR on which the doubly constrained differential corrections
    achieved convergence.
  \item $\mathcal{X}\coloneqq A\times R$ is the orbital elements space
    in attributable coordinates, where
    $A\coloneqq \mathbb{S}^1 \times (-\pi/2, \pi/2) \times \R^2$ is
    the attributable space and $R\coloneqq \R_{\geq 0}\times \R$ is
    the range/range-rate space.
  \item $\mathcal M$ is the Manifold Of Variations, a 2-dimensional
    submanifold of $\mathcal{X}$.
  \item $\R^m$ is the residuals space, whose dimension is $m\geq 6$
    since the number of observations must be $\ge 3$.
\end{enumerate}
We now define the maps we use to propagate back the density
$p_{\bm\Xi}$. First, the residuals are a function of the fit
parameters, that is $\bm \xi=F(\bold x)$, with
$F:\mathcal{X}\rightarrow \R^m$ being a differentiable map. The second
map goes from the AR space to the MOV space.
\begin{definition}
    The map $f_{\mu}:K'\rightarrow \mathcal{M}$ is defined to
    be
    \[
        f_{\mu}(\bm\rho) \coloneqq (\mathcal{A}^*(\bm\rho),\bm\rho),
    \]
    where $\mathcal{A}^*(\bm\rho)\in A$ is the best-fit attributable
    obtained at convergence of the doubly constrained differential
    corrections.
\end{definition}

\begin{lemma}\label{prop:MOVmanifold}
    The map $f_{\mu}$ is a global parametrization of $\mathcal{M}$ as
    a $2$-dimensional manifold.
\end{lemma}
\begin{proof}
    The set $K'$ is a subset of $\R^2$ with non-empty interior, the
    map $f_{\mu}$ is at least $C^1$ and its Jacobian matrix is
    \begin{equation}\label{eq:diff_f_M}
      (Df_\mu)_{\bm\rho} = \frac{\partial
        f_\mu}{\partial \bm\rho}(\bm\rho) =
      \begin{pmatrix}
        \dfrac{\partial\mathcal A^*}{\partial\bm\rho}(\bm\rho)\\[0.3cm]
        I_2
      \end{pmatrix},
    \end{equation}
    from which is clear that it is full rank on $K'$.
\end{proof}

\noindent The last map to introduce is that from the sampling space
$S$ to the AR space $R$.
\begin{definition}
    The map $f_\sigma:S\rightarrow R$ is defined according to the
    sampling technique:
    \begin{enumerate}
      \item if the sampling is uniform in $\rho$, $f_\sigma$ is the
        identity map;
      \item if the sampling is uniform in $\log_{10}\rho$, we have
        $S=\R^2$ and $f_\sigma (x,y) \coloneqq
        (10^x,y)$;
      \item if we are in the cobweb case,
        $S=\R_{\geq 0}\times \mathbb{S}^1$ and the map $f_\sigma$ is
        given by
        \[
            f_\sigma(r,\theta)\coloneqq r
            \begin{pmatrix}
                \sqrt{\lambda_1} \cos{\theta} & -\sqrt{\lambda_2} \sin{\theta}\\
                \sqrt{\lambda_2} \sin{\theta} & \phantom{-}\sqrt{\lambda_1} \cos{\theta}
            \end{pmatrix}
            \mathbf{v}_1 +
            \begin{pmatrix}
                \rho^* \\
                \dot{\rho}^*
            \end{pmatrix},
            \
        \]
        where $\lambda_1>\lambda_2$ are the eigenvalues of the
        $2 \times 2$ matrix $\Gamma_{\bm\rho\bm\rho}(\mathbf{x}^*)$,
        the latter being the restriction of the covariance matrix
        $\Gamma(\mathbf{x}^*)$ to the $(\rho,\dot\rho)$ space, and
        $\mathbf{v_1}$ is the unit eigenvector corresponding to
        $\lambda_1$.
    \end{enumerate}
\end{definition}

\noindent We then consider the following chain of maps
\[
    S \overset{f_\sigma}{\longrightarrow} R \supseteq K'
    \overset{f_{\mu}}{\longrightarrow} \mathcal{M}\subseteq
    \mathcal{X} \overset{F}{\longrightarrow}\R^m
\]
and we use it to compute the probability density function on $S$
obtained by propagating $p_{\bm\Xi}$ back.

\subsection{Conditional density of a Gaussian on an affine subspace}
\label{sec:gauss-cond}

In this section we establish general results about the conditional
probability density function of a Gaussian random variable on an
affine subspace. Let $m$ and $N$ be two positive integers, with
$m>N$. Let $B\in \Mat_{m,N}(\R)$ a $m\times N$ matrix with full rank,
that is $\rk(B)=N$. Consider the affine $N$-dimensional subspace of
$\R^m$ given by
\[
    W\coloneqq \left\{\bm\xi\in \R^m \,:\, \bm\xi = B\mathbf{x}+\bm\xi^*,\
    \mathbf{x}\in \R^N\right\} = \Imm(B) + \bm\xi^*.
\]
We can also assume that $\bm\xi^*$ is orthogonal to $W$, that is
$\bm\xi^*\in \Imm(B)^\perp$, otherwise it is possible to subtract the
component parallel to $W$. Given the random variable $\bm\Xi$ with the
Gaussian distribution on $\R^m$ as in \eqref{eq:res_norm_gauss}, we
want to find the conditional probability density $p_{\bm\Xi|W}$ of
$\bm\Xi$ on $W$. Let $R\in \Mat_m(\R)$ be a $m\times m$ rotation
matrix and let $f_R:\R^m\rightarrow \R^m$ be the affine map
\[
    f_R(\bm\xi) \coloneqq R(\bm\xi-\bm\xi^*).
\]
Throughout this section, we use the notation
$f_R(\bm\xi) \eqqcolon \bm\xi_R =
  \begin{psmallmatrix}\bm\xi'\\\bm\xi'' \end{psmallmatrix}$, with $\bm\xi'\in \R^{m-N}$ and $\bm\xi''\in \R^N$.
  We choose $R$ in such a way that
\begin{equation}
    \label{eq:rotation}
    f_R^{-1}\begin{pmatrix}\mathbf{0}\\\bm\xi''\end{pmatrix} =
    R^\top \begin{pmatrix}\mathbf{0}\\\bm\xi''\end{pmatrix} +
    \bm\xi^*\in W \quad\text{for all $\bm\xi''\in \R^N$}.
\end{equation}

\begin{lemma}
    \label{lemma:equivalence_rotation}
    Condition \eqref{eq:rotation} holds for all $\bm\xi''\in \R^N$ if
    and only if there exists an invertible matrix $A\in \Mat_N(\R)$
    such that
    $RB = \begin{psmallmatrix}0\\A\end{psmallmatrix}$.
\end{lemma}
\begin{proof}
    $(\Leftarrow)$ Let $\begin{psmallmatrix}\mathbf{0}\\
        \bm\xi''\end{psmallmatrix}\in \R^m$. Since $A$ is invertible
    there exists $\tilde{\mathbf{x}}\in\R^N$ such that
    $A\tilde{\mathbf{x}} = \bm\xi''$. Then
    $R^\top \begin{psmallmatrix}\mathbf{0}\\\bm\xi''\end{psmallmatrix}
    =
    R^\top \begin{psmallmatrix}0\\A\end{psmallmatrix}\tilde{\mathbf{x}}
    = R^\top RB\tilde{\mathbf{x}} = B\tilde{\mathbf{x}}$, hence
    $R^\top \begin{psmallmatrix}\mathbf{0}\\\bm\xi''\end{psmallmatrix}
    +
    \bm\xi^* = B\tilde{\mathbf{x}}+ \bm\xi^*\in W$.\\[0.15cm]
    $(\Rightarrow)$ Condition \eqref{eq:rotation} implies that for all
    $\bm\xi''\in \R^N$ there exists $\tilde{\mathbf{x}}\in\R^N$ such
    that
    $\begin{psmallmatrix}\mathbf{0}\\\bm\xi''\end{psmallmatrix} =
    RB\tilde{\mathbf{x}}$. Since the multiplication by $RB$ is
    injective, such $\tilde{\mathbf{x}}$ is unique. Therefore there
    exists a well-defined map $P:\R^N\rightarrow \R^N$ such that
    $P(\bm\xi'')=\tilde{\mathbf{x}}$. The map $P$ is linear and
    injective since $\Ker{P}=\{\mathbf{0}\}$, thus it is a
    bijection. If $A$ is the $N\times N$ matrix associated to
    $P^{-1}$, there easily follows
    $RB = \begin{psmallmatrix}0\\A\end{psmallmatrix}$.
    \end{proof}

\begin{lemma}
    \label{lemma:rotation_W}
    Let $U\coloneqq \{\bm\xi_R\in \R^m \,:\, \bm\xi' =
    \mathbf{0}\}$. Then the following holds:
    \begin{enumerate}[\upshape(i)]
      \item $U=f_R(W)=R\Imm{B}$ and the map
        $\left.f_R\right|_{W}:W\rightarrow U$ is a bijection;
      \item $R\bm\xi^*
        = \begin{psmallmatrix}\bm\xi'^*\\\mathbf{0}\end{psmallmatrix}$
        for some $\bm\xi'^*\in \R^{m-N}$.
    \end{enumerate}
\end{lemma}
\begin{proof}
    (i) The inclusion $U\subseteq f_R(W)$ is trivial. For the other
    one, let $\bm\xi\in W$, then there exists
    $\tilde{\mathbf{x}}\in \R^N$ such that
    $\bm\xi = B\tilde{\mathbf{x}} + \bm\xi^*$. From Lemma
    \ref{lemma:equivalence_rotation} we have
    $f_R(\bm\xi)=RB\tilde{\mathbf{x}}
    = \begin{psmallmatrix}0\\A\end{psmallmatrix} \tilde{\mathbf{x}}\in
    U$.
    \\[0.15cm]
    (ii) Since $\bm\xi^*\in (\Imm{B})^\perp$ and $R$ is an isometry,
    we have $R\bm\xi^*\in R(\Imm{B})^\perp=(R\Imm{B})^\perp=U^\perp$,
    where the last equality follows from (i).
\qed\end{proof}

\begin{theorem}
    \label{teo:cond_gauss}
    The conditional probability density of $\bm\Xi''$ is
    $N(\mathbf{0},I_N)$, that is    \[
    p_{\bm\Xi''}(\bm\xi'') = \frac{1}{(2\pi)^{N/2}} \exp\left(-\frac 12
    \bm\xi''^\top \bm\xi''\right)
    \]
\end{theorem}
\begin{proof}
  By the standard propagation formula for probability density
  functions under the action of a continuous function, we have that
    \begin{align*}
        p_{f_R(\bm\Xi)}(\bm\xi_R) &= \left|\det R\right| p_{\bm\Xi}(f_R^{-1}(\bm\xi_R)) = \\
        &= \frac{1}{(2\pi)^{m/2}} \exp\left(-\frac 12 (R^\top \bm\xi_R +
          \bm\xi^*)^\top(R^\top \bm\xi_R + \bm\xi^*)\right) =\\
        &= \frac{1}{(2\pi)^{m/2}} \exp\left(-\frac 12
          \left(\bm\xi_R^\top\bm\xi_R + 2\bm\xi_R^\top R\bm\xi^* +
            {\bm\xi^*}^\top\bm\xi^*\right)\right).
    \end{align*}
    The conditional probability density of the variable $f_R(\bm\Xi)$
    on $f_R(W)=U$ (see Lemma~\ref{lemma:rotation_W}-(i)) is obtained
    by using that $\bm\xi'=\mathbf{0}$ and
    Lemma~\ref{lemma:rotation_W}-(ii). We have
    \[
        p_{f_R(\bm\Xi)|f_R(W)}(\bm\xi'') =
        \frac{\frac{1}{(2\pi)^{m/2}} \exp\left(-\frac 12
            \left(\bm\xi''^\top\bm\xi''+{\bm\xi^*}^\top\bm\xi^*\right)\right)}{\int_{\R^N}
          \frac{1}{(2\pi)^{m/2}} \exp\left(-\frac 12
            \left(\bm\xi''^\top\bm\xi''+{\bm\xi^*}^\top\bm\xi^*\right)\right)\dd\bm\xi''}=
        \frac{1}{(2\pi)^{N/2}} \exp\left(-\frac 12 \bm\xi''^\top
          \bm\xi''\right),
    \]
    where we have used that $\int_{\R^n} \exp\left(-\frac 12
      \mathbf{y}^\top \mathbf{y}\right)\dd\mathbf{y} = (2\pi)^{n/2}$
    for all $n\geq 1$. Now the thesis follows by noting that the
    random variable $f_R(\bm\Xi)|f_R(W)$ coincides with $\bm\Xi''$.
\qed\end{proof}

\begin{corollary}
    \label{cor:cond_gauss}
    The conditional probability density of $\bm\Xi$ on $W$ is
    \[
    p_{\bm\Xi|W}(\bm\xi) = \frac{1}{(2\pi)^{N/2}}\exp\left(-\frac 12
    (\bm\xi-\bm\xi^*)^\top(\bm\xi-\bm\xi^*)\right),
    \]
    where $\bm\xi\in W$.
\end{corollary}
\begin{proof}
    From Lemma~\ref{lemma:rotation_W}-(i) we know that $f_R|_W$ is a
    bijection. Thus for all $\bm\xi''\in \R^N$ there exists a unique
    $\bm\xi\in W$ such that
    $\begin{psmallmatrix}\mathbf{0}\\\bm\xi''\end{psmallmatrix}=f_R(\bm\xi)=R(\bm\xi-\bm\xi^*)$. Now
    it suffices to use this fact in the equation for
    $p_{\bm\Xi''}(\bm\xi'')$ in Theorem~\ref{teo:cond_gauss}.
\qed\end{proof}

\subsection{Computation of the probability density}
\label{sec:pdf}

We first derive the probability density function on the Manifold Of
Variations from that on the normalised residuals space. The
computation is based on the linearisation of the map $F$ around the
best-fitting orbit. In Section~\ref{sec:prob_full_nonlin} we will
instead show the result of the full non-linear propagation of the
probability density function and discuss the reason for which this
approach, although correct, is not a suitable choice for applications.

\begin{theorem}\label{thm:res-mov}
    Let $\mathbf{X}$ be the random variable of $\mathcal{X}$. By
    linearising $F$ around the best-fitting orbit $\mathbf{x}^*$, the
    conditional probability density of $\mathbf{X}$ on
    $T_{\mathbf{x}^*}\mathcal{M}$ is given by
  \begin{equation}\label{eq:p_X}
    p_{\mathbf{X}|T_{\mathbf{x}^*}\mathcal{M}}(\mathbf{x}) =
    \dfrac{\displaystyle\exp\left(-\frac{\chi^2(\mathbf{x})}{2}\right)}
    {\displaystyle\int_{T_{\mathbf{x}^*}\mathcal{M}}
      \exp\left(-\frac{\chi^2(\mathbf{y})}{2}\right)\dd\mathbf{y}}.
  \end{equation}
\end{theorem}
\begin{proof}
    The map $F$ is differentiable of class at least $C^1$ and the
    Jacobian matrix of $F$ is the design matrix
    $B(\mathbf{x})=\frac{\partial F}{\partial
      \mathbf{x}}(\mathbf{x})\in \Mat_{m,N}(\R)$. Since the doubly
    constrained differential corrections converge to $\mathbf{x}^*$,
    the matrix $B(\mathbf{x}^*)$ is full rank. It follows that the map
    $F$ is a local parametrization of
    \[
        V\coloneqq F(\mathcal{M}) = \{\bm\xi\in \R^m\,:\,
        \bm\xi=F(\mathbf{x}),\ \mathbf{x}\in \mathcal{M}\},
    \]
    in a suitable neighbourhood of
    $\bm\xi^*\coloneqq F(\mathbf{x}^*)=\xi(\mathbf{x}^*)$. The set $V$
    is thus a 2-dimensional submanifold of the residuals space $\R^m$.
    Consider the differential
    \[
        DF_{\mathbf{x}^*}: T_{\mathbf{x}^*}\mathcal{M} \rightarrow
        T_{\bm\xi^*}V,
    \]
    where $T_{\bm\xi^*}V = \{\bm\xi\in \R^m \,:\,
    \bm\xi=\bm\xi^*+B(\mathbf{x}^*)(\mathbf{x}-\mathbf{x^*}),\,
    \mathbf{x}\in T_{\mathbf{x}^*}\mathcal{M}\}$ is a 2-dimensional
    affine subspace of $\R^m$. We claim that $\bm\xi^*$ is orthogonal to
    $T_{\bm\xi^*}V$: since $\mathbf{x}^*$ is a local minimum of the
    target function $Q$
    \[
        \mathbf{0} = \frac{\partial Q}{\partial \mathbf{x}}(\mathbf{x}^*)
        = \frac{2}{m}\bm\xi(\mathbf{x}^*)^\top B(\mathbf{x}^*),
    \]
    that is
    $\bm\xi(\mathbf{x}^*) = \bm\xi^* \in
    \Imm\left(B(\mathbf{x}^*)\right)^\perp$. By applying
    Corollary~\ref{cor:cond_gauss} we have that
    \[
        p_{\bm\Xi|T_{\bm\xi^*}V}(\bm\xi) = \frac{1}{2\pi}\exp\left(-\frac 12
          (\bm\xi-\bm\xi^*)^\top(\bm\xi-\bm\xi^*)\right)
    \]
    for $\bm\xi\in T_{\bm\xi^*}V$. The differential map is continuous
    and invertible (since it is represented by the matrix
    $A(\mathbf{x}^*)$, as in
    Lemma~\ref{lemma:equivalence_rotation}), thus we can use the
    standard formula for the transformations of random variables to
    obtain
    \[
        p_{\mathbf{X}|T_{\mathbf{x}^*}\mathcal{M}}(\mathbf{x}) =
        \frac{\sqrt{\det C(\mathbf{x}^*)}}{2\pi}\exp\left(-\frac 12
          (\mathbf{x}-\mathbf{x}^*)^\top
          C(\mathbf{x}^*)(\mathbf{x}-\mathbf{x}^*)\right),
    \]
    for $\mathbf{x}\in T_{\mathbf{x}^*}\mathcal{M}$, where
    $C(\mathbf{x}^*)$ is the $6\times 6$ normal matrix of the
    differential corrections converging to
    $\mathbf{x}^*$. Furthermore, in the approximation used, we have
    \[
        \chi^2(\mathbf{x}) =
        mQ(\mathbf{x})-mQ^*=(\mathbf{x}-\mathbf{x}^*)^\top
        C(\mathbf{x}^*)(\mathbf{x}-\mathbf{x}^*),
    \]
    so that for $\mathbf{x}\in T_{\mathbf{x}^*}\mathcal{M}$
    \[
        p_{\mathbf{X}|T_{\mathbf{x}^*}\mathcal{M}}(\mathbf{x}) =
        \frac{\sqrt{\det
            C(\mathbf{x}^*)}}{2\pi}\exp\left(-\frac{\chi^2(\mathbf{x})}{2}\right).
    \]
    Now the thesis follows by normalising the density just obtained.
\qed\end{proof}

We now complete the derivation of the probability density function on
the space $S$. In particular, to obtain the density on the space $R$
we use the notion of integration over a manifold because we want to
propagate a probability density on a manifold to a probability density
over the space that parametrizes the manifold itself (see
Theorem~\ref{thm:mov-ar}). Then the last step involves two spaces with
the same dimension, namely $R$ and $S$, thus we simply apply standard
propagation results from probability density functions (see
Theorem~\ref{thm:p_X}).

\begin{theorem}
    \label{thm:mov-ar}
    Let $\mathbf{R}$ be the random variable on the space
    $R$. Assuming~\eqref{eq:p_X} to be the probability density
    function on $\mathcal{M}$, the probability density function of
    $\mathbf{R}$ is
    \begin{equation*}
        p_{\mathbf{R}}(\bm\rho) = \dfrac{\displaystyle \exp\left(
            -\frac{\chi^2(\bm\rho)}{2}\right)
          \sqrt{G_{\mu}(\bm\rho)} } {\displaystyle\int_{K'}
          \exp\left(-\frac{\chi^2(\bm\rho)}{2}\right)
          \sqrt{G_{\mu}(\bm\rho)}\dd\bm\rho},
    \end{equation*}
    where $\chi^2(\bm\rho) = \chi^2(\mathbf{x}(\bm\rho))$ and
    $G_{\mu}$ is the Gramian determinant
    \[
        G_{\mu}(\bm\rho)=\det \left(I_2 + \left(\frac{\de
              \mathcal A^*}{\de\bm\rho}(\bm\rho)\right)^\top  \frac{\de
            \mathcal A^*}{\de\bm\rho}(\bm\rho)\right).
    \]
    Moreover, neglecting terms containing the second derivatives of
    the residuals multiplied by the residuals themselves, for all
    $\bm\rho\in K'$ we have
    \[
        \frac{\partial \mathcal A^*}{\partial \bm\rho}(\bm\rho) = -
        C_{\mathcal A}(\mathcal A^*(\bm\rho),\bm\rho)^{-1}B_{\mathcal
          A}(\mathcal A^*(\bm\rho),\bm\rho)^\top B_{\bm\rho}(\mathcal
        A^*(\bm\rho),\bm\rho).
    \]
\end{theorem}
\begin{proof}
    We have already proved that the map
    $f_{\mu}:K'\rightarrow \mathcal{M}$ is a global parametrization
    of $\mathcal{M}$. From the definition of integral of a function
    over a manifold we have
    \[
        p_{\mathbf{R}}(\bm\rho) =
        p_{\mathbf{X}|T_{\mathbf{x}^*}\mathcal{M}} (f_{\mu}(\bm\rho))
        \cdot \sqrt{\det \left[\frac{\partial f_{\mu}}{\partial
              \bm\rho}(\bm\rho)\right]^\top \frac{\partial
            f_{\mu}}{\partial \bm\rho}(\bm\rho)}
    \]
    and the thesis follows since from Equation~\eqref{eq:diff_f_M}
    \[
        \left[\frac{\partial f_{\mu}}{\partial
            \bm\rho}(\bm\rho)\right]^\top \frac{\partial
          f_{\mu}}{\partial \bm\rho}(\bm\rho) = I_2+
        \left(\dfrac{\partial\mathcal
            A^*}{\partial\bm\rho}(\bm\rho)\right)^\top
        \dfrac{\partial\mathcal A^*}{\partial\bm\rho}(\bm\rho).
    \]
    The equation for
    $\frac{\partial \mathcal A^*}{\partial \bm\rho}(\bm\rho)$ is
    proved in \cite{spoto:immimp}, Appendix A, but for the sake of
    completeness we repeat the argument here. Let $\bm\rho\in K'$. By
    definition, the point
    $\mathbf{x}(\bm\rho)=(\mathcal A^*(\bm\rho), \bm\rho)\in
    \mathcal{M}$ is a zero of the function
    $g:\mathcal{X}\rightarrow \R$ defined to be
    \[
        g(\mathbf{x})\coloneqq \frac{m}{2} \frac{\partial
          Q}{\partial\mathcal A}(\mathbf{x}) = B_\mathcal
        A(\mathbf{x})^\top \bm\xi(\mathbf{x}).
    \]
    The function $g$ is continuously differentiable and we have
    \[
        \frac{\partial g}{\partial \mathcal A}(\mathbf{x})
        =
        \frac{\partial}{\partial \mathcal A} \left(\frac{\partial
            \bm\xi}{\partial \mathcal A}(\mathbf{x})\right)^\top
        \bm\xi(\mathbf{x}) + \left(\frac{\partial \bm\xi}{\partial
            \mathcal A} (\mathbf{x})\right)^\top
        \frac{\partial\bm\xi}{\partial\mathcal A} (\mathbf{x})
        \simeq \left(\frac{\partial \bm\xi}{\partial \mathcal A}
          (\mathbf{x})\right)^\top
        \frac{\partial\bm\xi}{\partial\mathcal A} (\mathbf{x}) =
        C_{\mathcal A}(\mathbf{x}),
    \]
    where we neglected terms containing the second derivatives of the
    residuals multiplied by the residuals themselves. The matrix
    $C_{\mathcal A}(\mathbf{x}(\bm\rho))$ is invertible because
    $\bm\rho\in K'$, which means that the doubly constrained
    differential corrections did not fail. By applying the implicit
    function theorem, there exists a neighbourhood $U$ of $\bm\rho$, a
    neighbourhood $W$ of $\mathcal A^*(\bm\rho)$, a continuously
    differentiable function $\bold f:U\rightarrow W$ such that, for
    all $\widetilde{\bm\rho}\in U$ it holds
    \[
        g(\mathcal A^*,\widetilde{\bm\rho})=\bold 0 \Leftrightarrow
        \mathcal A^*= \bold f(\widetilde{\bm\rho}),
    \]
    and
    \begin{equation}\label{eq:Dini}
        \frac{\partial \bold f}{\partial \bm\rho}(\widetilde{\bm\rho}) =
        -\left(\frac{\partial g}{\partial \mathcal A} (\mathcal
          A^*(\widetilde{\bm\rho}),\widetilde{\bm\rho})\right)^{-1}
        \frac{\partial g}{\partial \bm\rho}(\mathcal
        A^*(\widetilde{\bm\rho}),\widetilde{\bm\rho}).
    \end{equation}
    The derivative $\frac{\partial g}{\partial \mathcal A}$ is already
    computed above. For
    $\frac{\partial g}{\partial \bm\rho}(\mathbf{x})$ we have
    \[
        \frac{\partial g}{\partial \bm\rho}(\mathbf{x}) =
        \frac{\partial}{\partial \bm\rho} \left(\frac{\partial \bm\xi}
          {\partial \mathcal A}(\mathbf{x})\right)^\top
        \bm\xi(\mathbf{x}) + \left(\frac{\partial \bm\xi}{\partial
            \mathcal A}(\mathbf{x})\right)^\top
        \frac{\partial\bm\xi}{\partial\bm\rho}(\mathbf{x})\simeq
        \left(\frac{\partial \bm\xi}{\partial \mathcal A}
          (\mathbf{x})\right)^\top
        \frac{\partial\bm\xi}{\partial\bm\rho}(\mathbf{x}) =
        B_{\mathcal A}(\mathbf{x})^\top B_{\bm\rho}(\mathbf{x}).
    \]
    The thesis now follows from Equation~\eqref{eq:Dini}.
\qed\end{proof}

The final step in the propagation of the probability density function
consists in applying the map $f_\sigma:S'\rightarrow K'$, where
$S'\coloneqq f_{\sigma}^{-1}(K')$ is the portion of the sampling space
mapped onto $K'$.

\begin{theorem}
  \label{thm:p_X}
  Let $\mathbf{S}$ be the random variable on the space
  $S$. Assuming~\eqref{eq:p_X}, the probability density function of
  $\mathbf{S}$ is
  \begin{equation}\label{eq:p_S}
    p_{\mathbf{S}}(\mathbf{s}) = \dfrac{\displaystyle \exp\left(
        -\frac{\chi^2(\mathbf{s})}{2}\right)
      \sqrt{G_{\mu}(\mathbf{s})} \sqrt{G_{\sigma}(\mathbf{s})}}
    {\displaystyle\int_{f_\sigma^{-1}(K')}
      \exp\left(-\frac{\chi^2(\mathbf{s})}{2}\right)
      \sqrt{G_{\mu}(\mathbf{s})} \sqrt{G_{\sigma}(\mathbf{s})} \dd\mathbf{s}},
  \end{equation}
  where we used the compact notation
  $\chi^2(\mathbf{s}) = \chi^2(\mathbf{x}(\bm\rho(\mathbf{s})))$,
  $G_{\mu}(\mathbf{s}) = G_{\mu}(\bm\rho(\mathbf{s}))$, and
  $G_{\sigma}$ is the Gramian of the columns of
  $Df_{\sigma}(\mathbf{s})$, so that
  \begin{equation*}
    \sqrt{G_{\sigma}(\mathbf{s})} = |\det Df_\sigma(\mathbf{s})|.
  \end{equation*}
\end{theorem}
\begin{proof}
    Equation~\eqref{eq:p_S} directly follows from the transformation
    law for random variables between spaces of the same dimension,
    under the action of the continuous function $f_\sigma$: it
    suffices to change the variables from the old ones to the new ones
    and multiply by the modulus of the determinant of the Jacobian of
    the inverse transformation, that is $f_{\sigma}$.  \qed\end{proof}

\noindent The determinant $\det Df_\sigma(\mathbf{s})$ depends on the
sampling technique and it is explicitly computed in
\cite{spoto:immimp}, Appendix A. In particular, the straightforward
computation yields the following:
\begin{enumerate}
  \item if the sampling is uniform in $\rho$ then
    $\det Df_\sigma(\mathbf{s})=1$ for all $\mathbf{s}\in S=\R_{\geq 0}\times \R$;
  \item if the sampling is uniform in the logarithm of $\rho$ then
    $\det Df_\sigma(\mathbf{s})=\log(10)\rho(\mathbf{s})$ for all
    $\mathbf{s}\in S=\R^2$;
  \item in the case of the spider web sampling
    $\det Df_\sigma(\mathbf{s}) = r\sqrt{\lambda_1\lambda_2}$ for all
    $\mathbf{s}\in S= \R_{\geq 0}\times \mathbb{S}^1$.
\end{enumerate}

For the sake of completeness, we now recall how the density
$p_{\mathbf{S}}$ is used for the computation of the impact probability
of a potential impactor. Each orbit of the MOV sampling is propagated
for 30 days, recording each close approach as a point on the Modified
Target Plane \citep{milani:ident2}, so that we are able to identify
which orbits are impactors. Let $\mathcal V\subseteq \mathcal{M}$ be
the subset of impacting orbits, so that
$f_{\sigma}^{-1}(f_{\mu}^{-1}(\mathcal{V}))\subseteq S$ is the
corresponding subset of sampling variables. The impact probability is
then computed as
\[
    \int_{\mathcal V}
    p_{\mathbf{S}}(\mathbf{s})\dd\mathbf{s}=\frac{\displaystyle
      \int_{f_{\sigma}^{-1}(f_{\mu}^{-1}(\mathcal{V}))}\exp\left(-\frac{\chi^2(\mathbf{s})}{2}\right)
      \sqrt{G_{\mu}(\mathbf{s})} \sqrt{G_{\sigma}(\mathbf{s})}
      \dd\mathbf{s}} {\displaystyle\int_{f_{\sigma}^{-1}(K')}
      \exp\left(-\frac{\chi^2(\mathbf{s})}{2}\right)
      \sqrt{G_{\mu}(\mathbf{s})} \sqrt{G_{\sigma}(\mathbf{s})}
      \dd\mathbf{s}}.
\]
Of course the above integrals are evaluated numerically as finite sums
over the integration domains: since they are subsets of $S$ we use the
already available sampling described in Section~\ref{sec:armov}. In
particular, we limit the sums to the sample points corresponding to
MOV orbits with a $\chi$-value less than $5$. As pointed out in
\cite{spoto:immimp}, this choice guarantees to NEOScan to find all the
impacting regions with a probability $>10^{-3}$, the so-called
completeness level of the system \citep{delvigna:compl_IM}.


\section{Full non-linear propagation}
\label{sec:prob_full_nonlin}

In this section we compute the probability density function on the
space $R$ obtained by a full non-linear propagation of the probability
density on the residuals space. In particular, the map $F$ is not
linearised around $\mathbf{x}^*$ as we did in
Theorem~\ref{thm:p_X}. This causes the inclusion in
Equation~\eqref{eq:p_X} of the contribution of the normal matrix $C$
as it varies along the MOV and not the fixed contribution
$C(\mathbf{x}^*)$ coming from the orbit $\mathbf{x}^*$ with the
minimum value of $\chi^2$. With the inclusion of all the non-linear
terms the resulting density of $\mathbf{R}$ has the same form as the
Jeffreys' prior and is thus affected by the same pathology discussed
in \cite{farnocchia2015}. This is the motivation for which we adopted
the approach presented in Section~\ref{sec:pdf}.

\begin{theorem}
    The probability density function of the variable $\mathbf{R}$
    resulting from a full non-linear propagation of the probability
    density of $\bm\Xi$ is
  \begin{equation*}
    p_{\mathbf{R}}(\bm\rho) = \dfrac{\displaystyle \exp\left(
        -\frac{\chi^2(\bm\rho)}{2}\right)
      \sqrt{\det C^{\bm\rho\bm\rho}(\bm\rho)}} {\displaystyle\int_{K'}
      \exp\left(-\frac{\chi^2(\bm\rho)}{2}\right)
      \sqrt{\det C^{\bm\rho\bm\rho}(\bm\rho)}\dd\bm\rho},
  \end{equation*}
  where $C^{\bm\rho\bm\rho} = \Gamma_{\bm\rho\bm\rho}^{-1}$ and
  $\Gamma_{\bm\rho\bm\rho}\in \Mat_2(\R)$ is the restriction of the
  covariance matrix $\Gamma$ to the $R$ space.
\end{theorem}
\begin{proof}
    The map $f_{\mu}:K'\rightarrow \mathcal{M}$ is a global
    parametrization of $\mathcal{M}$. From the properties of the map
    $F$ it is easy to prove that the map
    $F\circ f_{\mu} : K'\rightarrow F(\mathcal{M})$ is a global
    parametrization of the 2-dimensional manifold
    $V=F(\mathcal{M})$. From the notion of integration on a manifold
    we have
    \[
        p_{\mathbf{R}}(\bm\rho) = p_{\bm\Xi}(\bm\xi(\bm\rho))\cdot
        \sqrt{\det \left[\frac{\partial (F\circ f_{\mu})}{\partial
              \bm\rho}(\bm\rho)\right]^\top \frac{\partial (F\circ
            f_{\mu})}{\partial \bm\rho}(\bm\rho)}
    \]
    and by the chain rule
    \[
        \frac{\partial (F\circ f_{\mu})}{\partial
          \bm\rho}(\bm\rho) = \frac{\partial F}{\partial
          \mathbf{x}}(\mathbf{x}(\bm\rho))\frac{\partial
          f_{\mu}}{\partial \bm\rho}(\bm\rho) =
        B(\mathbf{x}(\bm\rho)) \begin{pmatrix}
            \dfrac{\partial\mathcal A^*}{\partial\bm\rho}(\bm\rho)\\[0.3cm]
            I_2
        \end{pmatrix},
    \]
    so that the Gramian matrix becomes
    \begin{align*}
        \left[\frac{\partial (F\circ
            f_{\mu})}{\partial \bm\rho}\right]^\top
        \frac{\partial (F\circ f_{\mu})}{\partial \bm\rho}
        &=
        \begin{pmatrix}
            \dfrac{\partial\mathcal A^*}{\partial\bm\rho}\\[0.3cm] I_2
        \end{pmatrix}^\top B^\top B \begin{pmatrix}
            \dfrac{\partial\mathcal
              A^*}{\partial\bm\rho}\\[0.3cm] I_2
        \end{pmatrix} =\\
        &= \left(\dfrac{\partial\mathcal
            A^*}{\partial\bm\rho}\right)^\top C_{\mathcal{A}\mathcal{A}}
        \dfrac{\partial\mathcal A^*}{\partial\bm\rho} +
        C_{\bm\rho\mathcal{A}} \dfrac{\partial\mathcal
          A^*}{\partial\bm\rho} + \left(\dfrac{\partial\mathcal
            A^*}{\partial\bm\rho}\right)^\top C_{\mathcal{A}\bm\rho} +
        C_{\bm\rho\bm\rho},
    \end{align*}
    where the matrices $C_{\mathcal{A}\mathcal{A}}=C_{\mathcal{A}}$,
    $C_{\mathcal{A}\bm\rho}$,
    $C_{\bm\rho\mathcal{A}}=C_{\mathcal{A}\bm\rho}^\top$, and
    $C_{\bm\rho\bm\rho}$ are the restrictions of the normal matrix
    $C(\mathbf{x}(\bm\rho))$ to the corresponding subspace. From
    Theorem~\ref{thm:mov-ar} we have
    \[
        \dfrac{\partial\mathcal A^*}{\partial\bm\rho} =
        -C_{\mathcal{A}}^{-1}B_{\mathcal{A}}^\top B_{\bm\rho} =
        -C_{\mathcal{A}}^{-1} C_{\mathcal{A}\bm\rho},
    \]
    so that the previous expression becomes
    \begin{align*}
        \left[\frac{\partial (F\circ
            f_{\mu})}{\partial \bm\rho}\right]^\top
        \frac{\partial (F\circ f_{\mu})}{\partial \bm\rho}
        &= C_{\mathcal{A}\bm\rho}^\top C_{\mathcal{A}}^{-1} C_{\mathcal{A}\bm\rho}
        -2 C_{\mathcal{A}\bm\rho}^\top C_{\mathcal{A}}^{-1} C_{\mathcal{A}\bm\rho} + C_{\bm\rho\bm\rho} = \\
        &=C_{\bm\rho\bm\rho}-C_{\mathcal{A}\bm\rho}^\top C_{\mathcal{A}}^{-1} C_{\mathcal{A}\bm\rho} = C^{\bm\rho\bm\rho} = \Gamma_{\bm\rho\bm\rho}^{-1},
    \end{align*}
    where the last equality is proved in \cite{milani:orbdet},
    Section~5.4.
\qed\end{proof}

\section{Conditional density on each attributable space}
\label{sec:prob-int}

In this section we prove that the conditional density of the
attributable $\mathcal{A}$ given $\bm\rho=\bm\rho_0\in K'$ is
Gaussian, providing a probabilistic interpretation to the doubly
constrained differential corrections.

Once $\bm\rho_0\in K'$ has been fixed, consider the fibre of
$\bm\rho_0$ with respect to the projection from $\mathcal{X}$ to $R$,
so that
\[
  H_{\bm\rho_0} \coloneqq A\times
  \{\bm\rho_0\}=\{(\mathcal{A},\bm\rho_0) \,:\, \mathcal{A}\in A\}.
\]
The fibre $H_{\bm\rho_0}$ is diffeomorphic to $A$ and thus
$H_{\bm\rho_0}$ is a $4$-dimensional submanifold of $\mathcal{X}$,
actually a 4-dimensional affine subspace, and the collection
$\{H_{\bm\rho_0}\}_{\bm\rho_0\in K'}$ is a 4-dimensional foliation of
$A\times K'\subseteq \mathcal{X}$. Theorem~\ref{thm:res_to_mov} gives
a probability density function on each leaf of this
foliation. Moreover, let us denote by
$\phi_{\bm\rho_0}:A\rightarrow H_{\bm\rho_0}$ the canonical
diffeomorphism between $A$ and $H_{\bm\rho_0}$, that is
$\phi_{\bm\rho_0}(\mathcal{A}) \coloneqq (\mathcal{A},\bm\rho_0)$ for
all $\mathcal{A}\in A$.

\begin{theorem}\label{thm:res_to_mov}
    Let $\mathbf{A}$ be the random variable of the space $A$. For each
    $\bm\rho_0\in K'$, the conditional probability density function of
    $\mathbf{A}$ given $\mathbf{R}=\bm\rho_0$ is
  \begin{align*}
    p_{\mathbf{A}|\mathbf{R}=\bm\rho_0}(\mathcal{A}) &=
        N\left(\mathcal{A}^*(\bm\rho_0),
        C_{\mathcal{A}}(\bm\rho_0)^{-1}\right)(\mathcal{A})=\\
       &=\frac{\sqrt{\det C_{\mathcal{A}}(\bm\rho_0)}}{(2\pi)^2}
                                         \exp\left(-\frac 12
                                         \left(\mathcal{A}-\mathcal{A}^*(\bm\rho_0)\right)^\top
                                         C_{\mathcal{A}}(\bm\rho_0)
                                         \left(\mathcal{A}-\mathcal{A}^*(\bm\rho_0)\right)\right),
  \end{align*}
  where we have used the compact notation
  $C_{\mathcal{A}}(\bm\rho_0)\coloneqq
  C_{\mathcal{A}}(\mathcal{A}^*(\bm\rho_0),\bm\rho_0)$.
\end{theorem}
\begin{proof}
    Define the map
    $G_{\bm\rho_0}\coloneqq F\circ \phi_{\bm\rho_0}:A\rightarrow
    \R^m$. The differential of $G_{\bm\rho_0}$ is represented by the
    design matrix $B_{\mathcal{A}}\in \Mat_{m,4}(\R)$ introduced in
    \eqref{C_A}. Consider the point
    $\mathbf{x}_0= (\mathcal{A}^*(\bm\rho_0),\bm\rho_0)\in
    H_{\bm\rho_0}$, where $\mathcal{A}^*(\bm\rho_0)$ is the best-fit
    attributable corresponding to $\bm\rho=\bm\rho_0$, that exists
    since $\bm\rho_0\in K'$. Given that the doubly constrained
    differential corrections converge to $\mathbf{x}_0$, the matrix
    $B_{\mathcal{A}}(\mathbf{x}_0)$ is full rank. It follows that the
    map $G_{\bm\rho_0}$ is a global parametrization of
    \[
        V_{\bm\rho_0}\coloneqq G_{\bm\rho_0}(A) = F(H_{\bm\rho_0})= \{\bm\xi\in \R^m
        \,:\, \bm\xi=F(\mathcal{A},\bm\rho_0),\ \mathcal{A}\in A\},
    \]
    that turns out to be a 4-dimensional submanifold of the residuals
    space $\R^m$, at least in a suitable neighbourhood of
    $\bm\xi_0\coloneqq F(\mathbf{x}_0) =
    G_{\bm\rho_0}(\mathcal{A}_0(\bm\rho_0))$. The map $G_{\bm\rho_0}$
    induces the tangent map between the corresponding tangent bundles
    \[
        DG_{\bm\rho_0}:TA\rightarrow TV_{\bm\rho_0}.
    \]
    In particular we consider the tangent application
    \[
        (DG_{\bm\rho_0})_{\mathcal{A}^*(\bm\rho_0)}:
        T_{\mathcal{A}^*(\bm\rho_0)}A \rightarrow T_{\bm\xi_0}V_{\bm\rho_0}.
    \]
    To use this map for the probability density propagation, we first
    need to have the probability density function on
    $T_{\bm\xi_0}V_{\bm\rho_0}$, that is an affine subspace of
    dimension 4 in $\R^m$. By Theorem~\ref{teo:cond_gauss} we have
    that the conditional probability density of $\bm\Xi$ on
    $T_{\bm\xi_0}V_{\bm\rho_0}$ is $N(\mathbf{0},I_4)$. Let $R$
    represent the rotation of the residuals space $\R^m$ such that
    condition \eqref{eq:rotation} holds, and let
    $A(\mathbf{x}_0)\in \Mat_4(R)$ as in Lemma
    \ref{lemma:equivalence_rotation}, so that the matrix
    $A(\mathbf{x}_0)^{-1}$ represents the inverse map
    $\left((DG_{\bm\rho_0})_{\mathcal{A}^*(\bm\rho_0)}\right)^{-1}$. By
    the transformation law of a Gaussian random variable under the
    linear map $A(\mathbf{x}_0)^{-1}$, we obtain a probability density
    function on the attributable space $A$ given by
    \[
        p_{\bf {A}}(\mathcal{A})= N(\mathcal{A}^*(\bm\rho_0),
        \Gamma_{\mathcal A}(\mathcal{A}^*(\bm\rho_0)))(\mathcal{A}),
    \]
    where $\mathbf{A}$ is the random variable on $A$ and
    \[
        \Gamma_{\mathcal A}(\mathcal{A}^*(\bm\rho_0)) =
        A(\mathbf{x}_0)^{-1}I_4(A(\mathbf{x}_0)^{-1})^\top =
        A(\mathbf{x}_0)^{-1}(A(\mathbf{x}_0)^{-1})^\top .
    \]
    As a consequence, the normal matrix of the random variable
    $\mathbf{A}$ is
    \[
        A(\mathbf{x}_0)^\top A(\mathbf{x}_0) =
        B_{\mathcal{A}}(\mathbf{x}_0)^\top R^\top RB_{\mathcal{A}}(\mathbf{x}_0) =
        B_{\mathcal{A}}(\mathbf{x}_0)^\top B_{\mathcal{A}}(\mathbf{x}_0) =
        C_{\mathcal{A}}(\mathbf{x}_0),
    \]
    which is in turn the normal matrix of the doubly constrained
    differential corrections leading to $\mathbf{x}_0$, computed at
    convergence. This completes the proof since $A$ and
    $A\times \{\bm\rho_0\}$ are diffeomorphic and thus the density
    $p_{\mathbf{A}}(\mathcal{A})$ is also the conditional density of
    $\mathbf{A}$ given $\mathbf{R}=\bm\rho_0$.
\qed\end{proof}


\section{Conclusions and future work}
\label{sec:conc}

In this paper we considered the hazard assessment of short-term
impactors based on the Manifold Of Variations method described in
\cite{spoto:immimp}. This technique is currently at the basis of
NEOScan, a software system developed at SpaceDyS and at the University
of Pisa devoted to the orbit determination and impact monitoring of
objects in the MPC NEO Confirmation Page. The aim of the paper was not
to describe a new technique, but rather to provide a mathematical
background to the probabilistic computations associated to the MOV.

Concerning the problem of the hazard assessment of a potential
impactor, the prediction of the impact location is a fundamental
issue. \cite{dimare:imp_corr} developed a semilinear method for such
predictions, starting from a full orbit with covariance. The MOV,
being a representation of the asteroid confidence region, could be
also used to this end and even when a full orbit is not available. The
study of such a technique and the comparison with already existing
methods will be subject of future research.


\bibliographystyle{elsarticle-harv} \bibliography{mov_ipcomp}

\end{document}